\documentclass[11pt]{article}
\usepackage[utf8]{inputenc}
\usepackage{fullpage}

\title{Combinatorial Mutations of Gelfand–Tsetlin Polytopes, Feigin–Fourier–Littelmann–Vinberg Polytopes, and \\ Block Diagonal Matching Field Polytopes}
\author{Oliver Clarke, Akihiro Higashitani, and Fatemeh Mohammadi }

\usepackage[square,sort,comma,numbers]{natbib}

\usepackage{amsthm}
\usepackage{amssymb}
\usepackage{amsmath}
\usepackage{pdfpages}
\usepackage{booktabs}
\usepackage{multirow}
\usepackage{graphics}
\usepackage{resizegather} 

\usepackage[bookmarksnumbered=true]{hyperref} 
\hypersetup{
     colorlinks = true,
     linkcolor = blue,
     anchorcolor = blue,
     citecolor = teal,
     filecolor = blue,
     urlcolor = blue
     }

\theoremstyle{plain}

\newtheorem{theorem}{Theorem}[section]
\newtheorem*{theoremIntro}{Theorem}

\newtheorem{proposition}[theorem]{Proposition}
\newtheorem{corollary}[theorem]{Corollary}
\newtheorem{lemma}[theorem]{Lemma}

\theoremstyle{definition}

\newtheorem{definition}[theorem]{Definition}
\newtheorem{example}[theorem]{Example}

\theoremstyle{remark}

\newtheorem{remark}[theorem]{Remark}

\newcommand{\RR}{\mathbb{R}}
\newcommand{\ZZ}{\mathbb{Z}}
\newcommand{\PP}{\mathbb{P}}
\newcommand{\QQ}{\mathbb Q}
\newcommand{\NN}{\mathbb N}
\newcommand{\CC}{\mathbb C}

\newcommand{\MB}{\mathcal{B}}
\newcommand{\MD}{\mathcal D}
\newcommand{\ML}{\mathcal L}

\newcommand{\conv}{{\rm Conv}}
\newcommand{\MO}{\mathcal O}
\newcommand{\MC}{\mathcal C}

\DeclareMathOperator{\GT}{GT}
\DeclareMathOperator{\FFLV}{FFLV}
\DeclareMathOperator{\rowspan}{rowspan}
\newcommand{\Gr}{{\rm Gr}}
\newcommand{\MP}{{\mathcal P}}
\newcommand{\init}{{\rm in}}

\begin{document}

\maketitle

\begin{abstract}
  The Gelfand-Tsetlin and the Feigin–Fourier–Littelmann–Vinberg polytopes for the Grassmannians are defined, from the perspective of representation theory, to parametrize certain bases for highest weight irreducible modules. These polytopes are Newton-Okounkov bodies for the Grassmannian and, in particular, the GT-polytope is an example of a string polytope. The polytopes admit a combinatorial description as the Stanley's order and chain polytopes of a certain poset, as shown by Ardila, Bliem and Salaza. We prove that these polytopes occur among matching field polytopes. Moreover, we show that they are related by a sequence of combinatorial mutations that passes only through matching field polytopes. As a result, we obtain a family of matching fields that give rise to toric degenerations for the Grassmannians. Moreover, all polytopes in the family are Newton-Okounkov bodies for the Grassmannians.
\end{abstract}

   {
  \hypersetup{linkcolor=black}
  \tableofcontents
}
\medskip

\section{Introduction}

We study the Gelfand-Tsetlin (GT) and Feigin-Fourier-Littelmann-Vinberg (FFLV) polytopes. The GT-polytope was introduced in \cite{Gelfand1950finite}, in the context of representation theory, to parametrise a basis for the irreducible representation $V(\lambda)$ for the Lie algebra $\mathfrak s \mathfrak l_n(\CC)$, with highest dominant integral weight $\lambda$. The FFLV-polytope was introduced in \cite{FFL11}, proving a conjecture of Vinberg \cite{vinberg2005some}, and parametrises a different basis for $V(\lambda)$. The lattice points of each polytope parametrise the respective bases for $V(\lambda)$. So it is immediate, from the perspective of representation theory, that these two polytopes have the same number of lattice points. Fourier asked whether there was a combinatorial reason for this \cite[Question~1.1]{ardilla2011GTandFFLVPolytopes} and Ardila, Bliem and Salaza gave a positive answer to this question by constructing the GT and FFLV-polytopes as the order and chain polytopes, respectively, of certain marked posets. Combinatorial techniques have proved successful in understanding GT-polytopes~\cite{FvectorGC, Alexandersson2016normalGTpolytopes}. In this paper, we study the GT and FFLV-polytopes through the lens of matching fields and combinatorial mutations that naturally arise from toric degenerations of the Grassmannian~$\Gr(k,n)$.

\smallskip
A toric degeneration of $\Gr(k,n)$ is a flat family over $\mathbb A^1_\CC$ such that the fiber over $0$ is a toric variety and all other fibers are isomorphic to $\Gr(k,n)$. There are many different approaches to constructing toric degenerations for the Grassmannian. These constructions arise from: representation theory through the use of standard monomial theory \cite{lakshmibai1998standard, kreiman2002richardson,KM05,kim2015richardson}, combinatorics via tropical geometry \cite{kaveh2019khovanskii,speyer2004tropical,KristinFatemeh,clarke2019toric}, and algebraic geometry via Newton-Okounkov bodies \cite{An13,rietsch2017newton,kaveh2019khovanskii, escobar2019wall}.

\smallskip

Two well-known families of toric degenerations are the Gelfand-Tsetlin (GT) \cite[Chapter~14]{MS05} and the Feigin-Fourier-Littelmann-Vinberg (FFLV) degenerations \cite{FFL15}. The polytopes associated to the toric fibers of these toric degenerations are the GT and FFLV-polytopes respectively. These polytopes are normal \cite{Alexandersson2016normalGTpolytopes, fang2016marked} and arise as from toric degenerations of the same variety, hence they have the same Ehrhart polynomial. This result is proved combinatorially in \cite{ardilla2011GTandFFLVPolytopes} by showing the order and chain polytopes of marked posets have the same Ehrhart polynomials. In this paper, we realise the GT-polytopes as matching field polytopes, which serve as the base case for the construction of a sequence of polytopes with the same Ehrhart polynomial, see Section~\ref{sec: intermediate matching field polytopes}.

\smallskip

Matching fields are combinatorial objects introduced by Sturmfels and Zelevinsky in the study of Newton polytopes of certain products of maximal minors \cite{sturmfels1993maximal}, see Section~\ref{sec: prelim matching fields and polytopes}. More recently, they have been shown to give rise to toric degenerations of Grassmannians \cite{KristinFatemeh,clarke2019toric,clarke2021combinatorial}, partial flag varieties \cite{OllieFatemeh2,clarke2022partialflag} and their Richardson varieties \cite{clarke2020standard,bonala2021standard,OllieFatemeh3}. Each matching field gives rise to a valuation on the Pl\"ucker algebra. We say that matching field gives rise to a toric degeneration if the valuation makes the Pl\"ucker coordinates into a Khovanskii basis \cite{kaveh2019khovanskii}. In this case, the matching field polytope is a Newton-Okounkov body for the Grassmannian, see Remark~\ref{rmk: NO-bodies}.

\smallskip

Each matching field $\Lambda$ additionally gives rise to a polytope $P_\Lambda \subseteq \RR^{k \times n}$. These polytopes are called matching field polytopes for the Grassmannian \cite{clarke2021combinatorial,clarke2022partialflag} or simply matching field polytopes. If a matching field gives rise to a toric degeneration then the projective toric variety associated to the polytope $P_\Lambda$ is the toric variety that appears as the fiber over $0$ in the corresponding toric degeneration. It turns out that properties of the matching field polytope can guarantee whether the matching field gives rise to a toric degeneration. In particular, by \cite[Theorem~1]{clarke2022partialflag}, if $P_\Lambda$ is combinatorial mutation equivalent to the GT-polytope, then $\Lambda$ gives rise to a toric degeneration.

\smallskip

A combinatorial mutation is a certain kind of piece-wise linear map, see Section~\ref{sec: prelim Mutations and Polytopes}. Mutations were originally introduced by Akhtar, Coates, Galkin, Kasprzyk in \cite{akhtar2012minkowski} to study mirror partners of $3$-dimensional Fano manifolds. Such mutations are defined as certain birational transformations of Laurent polynomials. A mutation of a Laurent polynomial induces a combinatorial mutation of its Newton polytope. In this paper, we define a combinatorial mutation in terms of its action on the dual polytope, see also \cite{clarke2021combinatorial,clarke2022partialflag}. Recently, mutations have been shown to relate: Newton-Okounkov bodies associated to adjacent prime cones of tropicalizations \cite{escobar2019wall}; families of polytopes for the flag variety \cite{higashitani2020two, fujita2020newton}; and families of matching field polytopes \cite{clarke2021combinatorial, clarke2022partialflag}. In this paper, we describe the GT and FFLV-polytopes for Grassmannians in terms of matching fields and show the following:

\begin{theoremIntro}[Theorem~\ref{thm:main_theorem}] There exist a sequence of combinatorial mutations taking the GT-polytope to the FFLV-polytope such that all intermediate polytopes are matching field polytopes.
\end{theoremIntro}

In particular, we show that the FFLV-polytope for the Grassmannian is a matching field polytope, see Theorem~\ref{thm: gr chain poset equals fflv polytope}, and as a result we have that all the matching fields, associated to the intermediate matching field polytopes, give rise to toric degenerations. 

\medskip

\noindent{\bf Outline.} In Section~\ref{sec: preliminaries}, we fix our notation and recall preliminary definitions and results for: combinatorial mutations in Section~\ref{sec: prelim Mutations and Polytopes}; matching fields and their polytopes in Section~\ref{sec: prelim matching fields and polytopes}; polytopes associated to posets in Section~\ref{sec: prelim two poset polytopes}; and toric degenerations on Grassmannians in Section~\ref{sec: prelim toric degen gr}, which includes a review of Gr\"obner degeneration and Khovanskii bases. The main results from the preliminaries are: Theorem~\ref{thm: Gr poset order polytope equals diagonal MF polytope}, the GT-polytope is a matching field polytope; and Theorem~\ref{thm: toric degenerations if mutation equivalent}, mutation equivalence preserves the property of a matching field giving rise to a toric degeneration. 

In Section~\ref{sec: FFLV polytopes for gr}, we study FFLV-polytopes for the Grassmannians using matching field polytopes. In Section~\ref{sec: fflv for gr3n}, we consider the case $\Gr(3,n)$ and show that the FFLV-polytopes are given by matching fields, see Example~\ref{ex: block diagonal matching field}. In Section~\ref{sec: fflv for grkn}, we define the FFLV matching field, see Definition~\ref{def: fflv weight vector and matching field}, and show that its polytope coincides with the FFLV-polytope, see Theorem~\ref{thm: gr chain poset equals fflv polytope}. In Section~\ref{sec: intermediate matching field polytopes}, we define a sequence of matching fields that are used in the proof of the main result Theorem~\ref{thm:main_theorem}, in Section~\ref{sec: fflv and gt mutation equiv}. To prove the result, we construct a sequence of mutations between the FFLV and GT-polytopes for the Grassmannian, through the matching field polytopes defined in Section~\ref{sec: intermediate matching field polytopes}.

\medskip

\noindent{\bf Acknowledgement.}~O.C. is an overseas researcher under Postdoctoral Fellowship of Japan Society for the Promotion of Science (JSPS). O.C. and F.M. were partially supported by the grants G0F5921N (Odysseus programme) and G023721N from the Research Foundation - Flanders (FWO), and the UGent 
BOF grant STA/201909/038. A.H. was partially supported by JSPS Fostering Joint International Research (B) 21KK0043
and JSPS Grant-in-Aid for Scientific Research (C) 20K03513.

\section{Preliminaries}\label{sec: preliminaries}

In this section, we review all preliminaries required for subsequent sections. We review combinatorial mutations \cite{akhtar2012minkowski} from the $M$-lattice perspective via piece-wise linear maps \cite{higashitani2020two}; matching fields, their polytopes, and the toric degenerations they induce \cite{KristinFatemeh, clarke2019toric, clarke2021combinatorial, clarke2022partialflag}; and the polytopes associated to a poset \cite{Stanley86twoposet}. In particular, we fix our notation and highlight important examples.

\subsection{Combinatorial mutations of lattice polytopes}\label{sec: prelim Mutations and Polytopes}

Throughout this section, we consider the vector space $\RR^d$ equipped with the standard inner-product $\langle \cdot, \cdot \rangle \colon \RR^d \times \RR^d \rightarrow \RR$. We will define combinatorial mutations of a lattice polytope following \cite{akhtar2012minkowski,clarke2021combinatorial,clarke2022partialflag}.

\begin{definition}\label{def: tropical map and mutation}
Fix a primitive lattice point $w \in \ZZ^d$ and a lattice polytope $F \subseteq w^\perp$. We define the \textit{tropical map} with data $w$ and $F$ as
\[\varphi_{w, F} : \RR^d \rightarrow \RR^d, \quad x \mapsto x - x_{\min} w
\]
where $x_{\min} = \min\{\langle x, f \rangle \colon f \in F \}$. Given a lattice polytope $P \subseteq \RR^d$, if $Q := \varphi_{w, F}(P)$ is lattice polytope then we say that $Q$ is a combinatorial mutation of $P$. We say that two lattice polytopes are \textit{mutation equivalent} if there exists a sequence of mutations from one to the other.
\end{definition}

The tropical map $\varphi_{w, F}$ is a piece-wise linear map. Moreover, for each region of linearity, the map is given by a unimodular map $x \mapsto x - \langle f, x \rangle w$ for some vertex $f \in F$. 

\begin{remark}\label{rmk: toric geometry M and N lattice convention}
Throughout the paper we consider lattice polytopes contained in $\RR^d$. That is a polytope whose vertices lie in the lattice $\ZZ^d \subseteq \RR^d$. From the perspective of toric geometry, this lattice is the $M$-lattice of characters of the torus, as opposed to its dual lattice $N$ of one-parameter subgroups. The tropical map in Definition~\ref{def: tropical map and mutation} is a map $\varphi_{w,F} : M_\RR \rightarrow M_\RR$ where $w \in M$ and $F \subseteq N_\RR$. So tropical maps act on the moment polytope $P$ of a toric variety whose fan is given by the inner normal fan of $P$. 
\end{remark}

\begin{example}\label{ex: mutation}
Let $w = (1,1) \in \RR^2$ and $F = \conv\{(0,0), (1,-1) \} \subseteq w^{\perp}$. Let $H = \langle w \rangle \subseteq \RR^2$ be the line spanned by $w$. The tropical map $\varphi_{w,F}$ is a piecewise shear given by
\[
\varphi_{w,F}(x,y) = \left\{\!\!
\begin{array}{ll}
    (y, 2y-x) &  \text{if } y \ge x,\\
    (x,y) & \text{otherwise.}
\end{array}
\right.
\]
Consider the plane hexagon $P \subseteq \RR^2$ given by 
\[
P = \conv\{(1,0), (0,1), (1,1), (-1, 0), (0, -1), (-1, -1) \}.
\]
The hexagon $P$ and its image under the tropical map are shown in Figure~\ref{fig: ex mutation}. From the diagram, we see that $\varphi_{w,F}$ fixes all points below the line $H$ and acts as a shear above $H$. In particular, the image is a lattice polygon so $\varphi_{w,F}$ defines a combinatorial mutation of $P$. 

\begin{figure}
    \centering
    \includegraphics{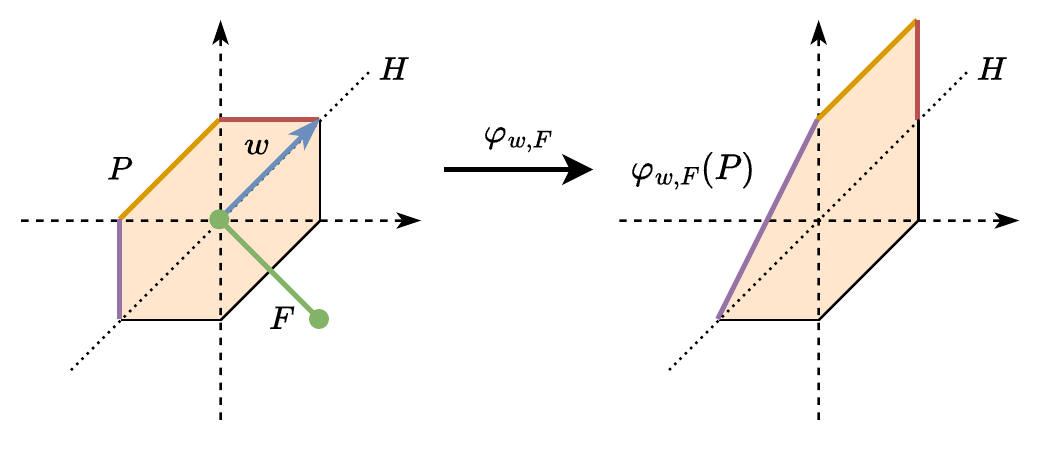}
    \caption{The combinatorial mutation in Example~\ref{ex: mutation}.}
    \label{fig: ex mutation}
\end{figure}
\end{example}

Given a lattice polytope $P \subseteq \RR^d$, its \textit{Ehrhart polynomial} $E_P(t) \in \QQ[t]$ is the function satisfying $E_P(n) = |nP \cap \ZZ^d|$ for all $n \in \ZZ_{\ge 0}$. See \cite{Ehrhart62polynomial, beck05coeffEhrhart} for more details. Combinatorial mutation preserves properties of lattice polytopes. In particular, we have the following.

\begin{proposition}[{\cite[Proposition~4]{akhtar2012minkowski}}]
If two lattice polytopes are mutation equivalent then they have the same Ehrhart polynomial.
\end{proposition}

\subsection{Matching fields and their polytopes}\label{sec: prelim matching fields and polytopes}
A matching field is a combinatorial object that has been used to successfully parametrise families of toric degenerations of Grassmannians. Here, we describe matching fields and their polytopes.

\smallskip

Given integers $1 \le k < n$, a \emph{matching field} for the Grassmannian $\Gr(k,n)$ is a map $\Lambda \colon \binom{[n]}{k} \rightarrow S_k$ which sends each $k$-subset $I \subseteq [n]$ to a permutation $\Lambda(I) \in S_k$. The permutation induces an order on the elements of $I = \{i_1 < \dots < i_k \}$ given by the tuple $(i_{\Lambda(I)(1)}, i_{\Lambda(I)(2)}, \dots, i_{\Lambda(I)(k)})$, which we call a \textit{tuple of $\Lambda$}. The set of all tuples determine a matching field uniquely. Therefore, we may define a matching field by its tuples and we identify $\Lambda$ with its collection of tuples. The entries of a tuple are commonly written vertically in a tableau.

\smallskip

Fix $k < n$ and a matching field $\Lambda$ for $\Gr(k,n)$. We denote by $\{e_{i,j} \in \RR^{k \times n} \colon i \in [k], \ j \in [n]\}$ the set of standard basis vectors. The matching field $\Lambda$ gives rise to a lattice polytope given by
\[
P_\Lambda = \conv\left\{ v_{\Lambda, \tilde J} := \sum_{i = 1}^k e_{i, j_i} \colon J = (j_1, j_2, \dots, j_k) \in \Lambda \right\} \subseteq \RR^{k \times n}
\]
where $\Lambda$ is taken as a collection of tuples and $\tilde J$ is the underlying set of the tuple $J$. Often, when the matching field is clear, we will write $v_I$ for the point $v_{\Lambda, I}$, for each subset $I \subseteq [n]$. Since each point $v_I$ is a $0/1$-vector containing exactly $k$-ones, it follows that the vertices of $P_\Lambda$ are given by $V(P_\Lambda) = \left\{v_I \colon I \in \binom{[n]}{k} \right\}$.~By abuse of notation, we will identify the tuples of $\Lambda$ with the vertices of~$P_\Lambda$.

\begin{definition}\label{def: coherent matching field}
Fix $k < n$ and let $M = (m_{i,j}) \in \RR^{k \times n}$ be a weight matrix. We say that $M$ is \textit{generic} if for each $J = \{j_1 < \dots < j_k \} \in \binom{[n]}{k}$ the minimum $\widehat M(J) := \min \left\{\sum_{i = 1}^k m_{i, j_{\sigma(i)}} \colon \sigma \in S_k\right\}$ is achieved by a unique permutation $\sigma_J \in S_k$. The \textit{matching field induced by $M$} is $\Lambda_M$, which is defined by $J \mapsto \sigma_J$. A matching field $\Lambda$ for $\Gr(k,n)$ is \textit{coherent} if $\Lambda = \Lambda_M$ for some generic weight matrix $M \in \RR^{k\times n}$.
\end{definition}

\begin{example}[Diagonal matching field]\label{ex: diagonal matching field}
Fix $k < n$. The \textit{diagonal matching field} is the matching field $\MD(k,n)$ which sends each $k$-subset to the identity permutation. The tuples of $\MD(k,n)$ are $(i_1 < i_2 < \dots < i_k)$ and so the vertices of the matching field polytope $P_{\MD(k,n)}$ are given by $e_{1,i_1} + e_{2, i_2} + \dots + e_{k, i_k} \in \RR^{k \times n}$. The matching field $\MD(k, n)$ is coherent since it is induced by the generic weight matrix $M_D$ given by $(M_D)_{1,j} = 0$ for $j \in [n]$ and $(M_D)_{i,j} = (n - j)n^{i-2}$ for each $i \in \{ 2, \dots, n\}$ and $j \in [n]$. For example, for $\Gr(3,6)$ this weight matrix is given by
\[
M_D = 
\begin{bmatrix}
0 & 0 & 0 & 0 & 0 & 0 \\
5 & 4 & 3 & 2 & 1 & 0 \\
30 & 24 & 18 & 12 & 6 & 0
\end{bmatrix}.
\]
\end{example}

Note that, not all matching fields are coherent.

\begin{example}[A non-coherent matching field]\label{ex: not coherent matching field}
The matching field $\Lambda = \{12,\  23,\  31 \}$, given as a collection of tuples, is not coherent. Assume by contradiction that $\Lambda$ is induced by some weight matrix
\[
M = \begin{bmatrix}
a & b & c \\
d & e & f
\end{bmatrix}
\in \RR^{2 \times 3}.
\]
Since $12$ is a tuple, we have that $a + e < b + d$. Similarly, $23$ is a tuple so we have $b + f < c + e$. These inequalities give $a + f < c + d$. However, the tuple $31$ implies that $c + d < a + f$, a contradiction.
\end{example}

We recall the following family of matching fields which have been studied in the context of toric degenerations \cite{KristinFatemeh, clarke2021combinatorial}.

\begin{example}[Block-diagonal matching field]\label{ex: block diagonal matching field}
Fix $k < n$. The \textit{block-diagonal matching field $\Lambda$} is the matching field induced by the weight matrix 
\[
M_\Lambda = 
\begin{bmatrix}
 0 & 0 & 0 & \dots & 0 \\
 0 & n-1 & n-2 & \dots & 1 \\
 \vdots & \vdots & \vdots & \ddots & \vdots 
\end{bmatrix}
\]
where the third and subsequent rows are equal to those of $M_D$. For each subset $J = \{j_1 < \dots < j_k \}$, the corresponding tuple of $\Lambda$ is given by $(\boldsymbol{j_2}, \boldsymbol{j_1}, j_3, j_4, \dots, j_k)$ if $1 \in J$, otherwise it is given by $(\boldsymbol{j_1}, \boldsymbol{j_2},j_3,j_4, \dots, j_k)$ if $1 \notin J$.
\end{example}

\begin{remark}
The block-diagonal matching fields have be studied in more generality. In \cite{clarke2022partialflag}, the authors study matching fields induced by a weight matrix obtained by permuting the second row of $M_D$, see Example~\ref{ex: diagonal matching field}. In particular, all such matching fields polytopes are mutation equivalent to the GT-polytope, hence each matching field gives rise to a toric degeneration of the Grassmannian. For our purposes, it suffices to consider the following weaker version of this result.
\end{remark}

\begin{theorem}[{\cite[Theorems~2 and 4]{clarke2021combinatorial}}]\label{thm: block diagonal comb mutation equiv}
Let $\Lambda$ be the block diagonal matching field for $\Gr(k,n)$. The matching field polytopes $P_\Lambda$ and $P_\MD$ are mutation equivalent. Moreover, the sequence of mutation may be chosen to pass only through matching field polytopes.
\end{theorem}

In this paper, we reinterpret this result in terms of order and chain polytopes associated to the Grassmannian poset for Grassmannian $\Gr(3,n)$, see Section~\ref{sec: fflv for gr3n}.

\subsection{Order and chain polytopes}\label{sec: prelim two poset polytopes}

In this section, we recall the definitions of two polytopes associated to a poset by Stanley \cite{Stanley86twoposet}, and recall the description of their vertices in terms of basic properties of posets.

\begin{definition} \label{def: order and chain polytope}
Fix a finite poset $\Pi = \{p_1, \dots, p_d\}$. We define the two polytopes
\begin{align*}
    \MO(\Pi) &= \{ x \in \RR^d \colon x_i \le x_j \text{ if } p_i < p_j \text{ and } 0 \le x_i \le 1 \text{ for all } i, j \in [d] \}, \\
    \MC(\Pi) &= \{ x \in \RR^d \colon x_{i_1} + \dots + x_{i_s} \le 1 \text{ if } p_{i_1} < \dots < p_{i_s} \text{ and } 0 \le x_i \text{ for all } i \in [d] \}
\end{align*}
called the \textit{order polytope} and \textit{chain polytope}, respectively.
\end{definition}

The vertices of the order and chain polytopes can be described in terms of the properties of the underlying poset. Let $(\Pi, <)$ be a poset.
A subset $A \subseteq \Pi$ is an \textit{anti-chain} if any pair of elements of $A$ are incomparable. A subset $F \subseteq \Pi$ is called a \textit{filter} if for any $p \in F$ and $q \in \Pi$ such that $p < q$, we have that $q \in F$. We note that the empty set is both an anti-chain and a filter.

\begin{proposition}[{\cite[Corollary~1.3 and Theorem~2.2]{Stanley86twoposet}}]\label{prop: vertices of order and chain polytopes}
Let $\Pi = \{p_1, \dots, p_d \}$ be a finite poset. The vertices of $\MO(\Pi)$ are in bijection with the filters of $\Pi$. Each filter $F \subseteq \Pi$ corresponds to the vertex $\chi_F = (x_1, \dots, x_d)$ the characteristic vector of $F$ where $x_i = 1$ if $p_i \in F$ and $x_i = 0$ if $p_i \notin F$. The vertices of the chain polytope $\MC(\Pi)$ are the characteristic vectors of the anti-chains of $\Pi$.
\end{proposition}

In particular, $\MO(\Pi)$ and $\MC(\Pi)$ are $0/1$-polytopes. Observe that for any filter $F \subseteq \Pi$, the set of elements $A(F) = \min F := \{x \in F \colon y \nless x \text{ for all } y \in F\}$ is an anti-chain. It is straightforward to show that the map $F \mapsto A(F)$ is a bijection between filters and anti-chains. This means that $\MO(\Pi)$ and $\MC(\Pi)$ have the same number of vertices. Moreover, the polytopes $\MO(\Pi)$ and $\MC(\Pi)$ have the same Ehrhart polynomial \cite{higashitani2020two}.

In this paper, we consider the polytopes associated to the Grassmannian poset, which we describe in the following important definition.

\begin{definition}[Grassmannian poset]\label{ex: grassmannian poset}
Fix $k < n$. Let $Q_{k,n} = \{(i,j) \in \ZZ^2 : 1 \le i \le k,\ 1 \le j \le n-k\}$ be a set of pairs of integers. We define a partial order on $Q_{k,n}$ given by
\[
(a,b) < (c,d) \iff a \le c\text{ and }\ b \le d.
\]
\end{definition}

\begin{example}\label{ex: Q37 poset}
The Hasse diagram for the poset $Q_{3,7}$ is shown in Figure~\ref{fig:gr37_poset}. The vertices of the order polytope $\MO(Q_{3,7})$ are in bijection with the filters of $Q_{3,7}$. For example, the filter $F = \{(3,2), (3,3), (3,4), (2,4), (1,4) \}$, the shaded elements in the figure, gives rise to the vertex $\chi_F$ which is the characteristic vector of $F$. Each filter can be identified by its corresponding anti-chain $\min F$. For this example, $\min F = \{ (3,2), (1,4)\}$, which are the square elements in the figure.

\begin{figure}
    \centering
    \includegraphics{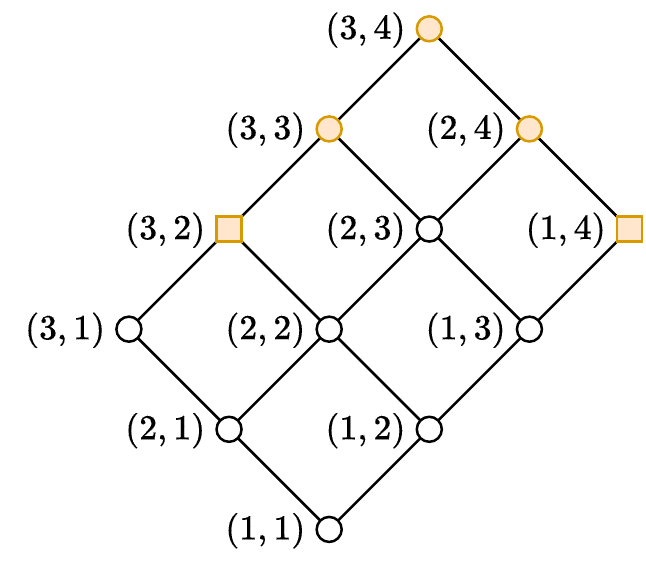}
    \caption{The Hasse diagram $Q_{3,7}$ in Example~\ref{ex: Q37 poset}. The shaded elements form a filter. The squares are the anti-chain of minimal elements of the filter.}
    \label{fig:gr37_poset}
\end{figure}
\end{example}

The order polytope $\MO(Q_{k,n})$ of the Grassmannian poset may be naturally identified with the matching field polytope $P_\MD$ for the diagonal matching field. For each filter $F \subseteq Q_{k,n}$ and $i \in [k]$ define $s_F(i) := |\{(i, j) \in F \colon j \in [n-k] \}| \in \{0, 1, \dots n-k \}$ and $J(F) := \{i + s_F(i) \colon i \in [k] \} \subseteq [n]$.

\begin{theorem}\label{thm: Gr poset order polytope equals diagonal MF polytope}
Fix $k < n$. The order polytope of the Grassmannian poset $\MO(Q_{k,n})$ is unimodular equivalent to the matching field polytope $P_\MD$ for the diagonal matching field. Moreover, the unimodular map between them may be chosen so that the vertex $\chi_F \in \MO(Q_{k,n})$, for some filter $F$, is mapped to the vertex $v_{\MD, J(F)} \in P_\MD$.
\end{theorem}

\begin{proof}

We write $x_{i,j} \in \RR^{k \times (n-k)}$ and $y_{i,j} \in \RR^{k \times n}$ for the standard basis vectors. We define the linear map
\[
\phi :  \RR^{k \times (n-k)} \rightarrow \RR^{k \times n}, \quad x_{i,j} \mapsto y_{i, (n-k-j)+i+1 } - y_{i, (n-k-j)+i}
\]
and the affine linear map $\widehat \phi(x) = \phi(x) + v_{\MD, \{1,2,\dots, k \}}$. We proceed to prove that $\widehat \phi(\MO(Q_{k,n})) = P_\MD$. Since $\widehat \phi$ is an affine linear map, it preserves convexity. Hence, it suffices to show that $\widehat \phi$ restricts to a bijection between the vertices of $\MO(Q_{k,n})$ and $P_\MD$. Fix a vertex $\chi_F \in \MO(Q_{k,n})$, where $F \subseteq Q_{k,n}$ is a filter. Since $F$ is a filter, if $(i,j) \in F$ then it follows that $(i, j+1), (i, j+2), \dots, (i, n-k) \in F$. Therefore
\[
\phi(\chi_F) = 
\sum_{(i,j) \in F} \phi(x_{i,j}) = 
\sum_{(i,j) \in F} y_{i, (n-k+i)-j+1} - y_{i, (n-k+i)-j} = \sum_{\substack{(i,j) \in F \\ j \text{ minimal}}} y_{i, (n-k+i)-j+1} - y_{i,i}.
\]
So $\widehat \phi(\chi_F) = \sum y_{i, (n-k+i)-j+1}$ where the sum is taken over $(i,j) \in \widehat{J(F)}$ where
\[
\widehat{J(F)} =  
\{(i,j) \in F \colon i \in [k], \ j \text{ is minimal}\} \cup 
\{(i, n-k+1) \colon i \in [k], \ (i,j) \notin F \text{ for any } j \in [n-k]\}.
\]
For each $i \in [k]$ we have that $\widehat{J(F)}$ contains exactly one element $(i,j)$ for some $j \in [n-k+1]$. By definition, it immediately follows that $s_F(i) = (n-k+1) - j$. Therefore $\widehat \phi(\chi_F) = \sum_{i \in [k]} y_{i, i+s_F(i)}$. Since $F$ is a filter, if $(i,j) \in F$ then we have $(i+1, j) \in F$. It follows that $s_F(1) \le s_F(2) \le \dots \le s_F(k)$ and so $\widehat \phi(\chi_F) = v_{\MD, J(F)}$. And so $\widehat \phi$ is a well-defined map between the vertices of $\MO(Q_{k,n})$ and $P_\MD$ that sends a vertex $\chi_F \in \MO(Q_{k,n})$ to $v_{\MD, J(F)} \in P_\MD$. The restriction of $\widehat \phi$ to the vertices of the polytopes has an inverse. Explicitly, for each subset $J = \{j_1 < \dots < j_k\} \subseteq [n]$, the inverse sends the vertex $v_{\MD, J} \in P_\MD$ to the vertex $\chi_{F(J)} \in \MO(Q_{k,n})$ where the filter $F(J) \subseteq Q_{k,n}$ is defined by its minimal elements
\[
\min(F(J)) := \{(i, (n-k+1)-(j_i-i))  \in Q_{k,n} \colon i \in [k], \ j_i-i > 0  \}.
\]
Therefore, the map $\widehat \phi$ defines a bijection on between the vertices of the polytopes $\MO(Q_{k,n})$ and $P_\MD$ that sends a vertex $\chi_F \in \MO(Q_{k,n})$ to $v_{\MD, J(F)} \in P_\MD$.

\medskip

We remark that $\phi$ is a unimodular map. By unimodular, we mean that $\phi$ restricts to a bijection between the distinguished lattices of $\RR^{k \times (n-k)}$ and $\RR^{k \times n}$. These lattices are defined to be the $\ZZ$-affine span of the vertices of the polytopes $\MO(Q_{k,n})$ and the translate $\widetilde{P_\MD}$ of the polytope $P_\MD$ by the vector $-v_{D,\{1, \dots, k\}}$. We observe that the $\ZZ$-affine span of the vertices of $\MO(Q_{k,n})$ is equal to the natural lattice $\ZZ^{k \times (n-k)} \subseteq \RR^{k \times (n-k)}$. Since $\phi$ defines a linear isomorphism between $\MO(Q_{k,n})$ and $\widetilde{P_\MD}$, it follows immediately that $\phi$ restricts to a bijection between their respective lattices. Explicitly, the distinguished lattice inside $\RR^{k \times n}$ is given by $\ZZ[y_{i,j} - y_{i,i}]_{i \in [k], j \in \{i+1, i+2, \dots, n-k+i\}}$. Hence, the polytopes $\MO(Q_{k,n})$ and $P_\MD$ are unimodular equivalent.
\end{proof}

We call the polytope $\MO(Q_{k,n})$ the GT-polytope (Gelfand-Tsetlin polytope) for the Grassmannian and identify it with the diagonal matching field polytope $P_\MD$.

\subsection{Toric degenerations of Grassmannians}\label{sec: prelim toric degen gr}

In this section, we show how coherent matching fields give rise to toric degenerations of the Grassmannians. Throughout, we fix $k < n$, a generic weight matrix $M \in \RR^{k \times n}$ and the coherent matching field $\Lambda = \Lambda_M$ induced by $M$, see Definition~\ref{def: coherent matching field}.

\medskip

\noindent \textbf{Grassmannians.} The Grassmannian $\Gr(k,n)$ is the space of $k$-dimensional linear subspaces of $\CC^n$. The Pl\"ucker embedding realises $\Gr(k,n)$ as a projective variety given by the image of the map
\[
\phi \colon \Gr(k,n) \rightarrow \PP^{\binom{n}{k} - 1}, \quad \rowspan
\begin{pmatrix}
x_{1,1} & \dots & x_{1,n} \\
\vdots & \ddots & \vdots \\
x_{k,1} & \dots & x_{k,n}
\end{pmatrix} \mapsto
(\det(X_I))_{I \in \binom{[n]}{k}}
\]
where $X = (x_{i,j})$ is a $k \times n$ matrix and $X_I$ is the $k \times k$ submatrix consisting of the columns indexed by $I\subseteq [n] := \{1,2, \dots, n\}$. We identify $\Gr(k,n)$ with its image under $\phi$. So $\Gr(k,n)$ is a subvariety of $\PP^{\binom n k -1}$ given by the vanishing locus of the ideal $G_{k,n} \subseteq \CC[P_I]_{I \in \binom{[n]}k}$. Explicitly, $G_{k,n}$ is the kernel of the polynomial map
\[
\phi^* \colon \CC[P_I]_{I \in \binom{[n]}k} \rightarrow \CC[x_{i,j}]_{i \in [k], j \in [n]}, \quad P_I \mapsto \det(X_I)
\]
where $X = (x_{i,j})$ is a $k \times n$ matrix of variables and $X_I$ is the submatrix of $X$ as above. We call $G_{k,n}$ the \textit{Pl\"ucker ideal} of the Grassmannian and each maximal minor $\det(X_I)$ a \textit{Pl\"ucker form}. The ring generated by the Pl\"ucker forms is called the \textit{Pl\"ucker algebra}.

\medskip
\noindent \textbf{Gr\"obner degenerations.} The theory of Gr\"obner fans, introduced by Mora and Robbiano \cite{mora1988grobner}, is one of the main tools in commutative algebra to degenerate polynomial ideals into toric ideals.
Let $w = (w_1, \dots, w_n)\in \RR^n$ be a weight vector for the polynomial ring $R = \CC[x_1, \dots, x_n]$. For each polynomial $f = \sum_{u \in \NN^n} c_u x^u \in R$, its weight with respect to $w$ is given by $w(f) = \min\{ u \cdot w \colon c_u \neq 0\}$ and its initial form with respect to $w$ is
\[
\init_w(f) = \sum_{\substack{u \in \NN^n \\ u \cdot w = w(f)}} c_u x^u.
\]
For each ideal $I \subseteq R$ associated to a variety $X \subseteq \CC^n$ we define its initial ideal with respect to $w$ as $\init_w(I) = \langle \init_w(f) \colon f \in I\rangle$. For each weight $w$, we obtain a flat family over $\mathbb A^1_{\CC}$ whose fibers are given by the ideals
\[
I_t = 
\left\langle 
t^{-w(f)} f(t^{w_1} x_1, t^{w_2} x_2, \dots, t^{w_n} x_n) \colon f \in I
\right\rangle
\]
for each $t \neq 0$ and $I_0 = \init_w(I)$. In particular, if $\init_w(I)$ is a toric ideal,~i.e.~a~binomial~prime~ideal, then we obtain a toric degeneration of $X$. See \cite{bossinger2021families} for a more general family of Gr\"obner degenerations.

\medskip

\noindent \textbf{Khovanskii bases.} Let $\{f_1, \dots, f_s\} \subseteq R = \CC[x_1, \dots, x_n]$ be a collection of polynomials and $w \in \RR^n$ a weight vector for $R$. Let $A = \CC[f_1, \dots, f_s] \subseteq R$ be a subalgebra of $R$. We define the initial algebra of $A$ to be
\[
\init_w(A) = \CC[\init_w(f) \colon f \in A]
\]
which is the subalgebra generated by the initial terms of elements of $A$. In general, we have $\CC[\init_w(f_1), \dots, \init_w(f_s)] \subseteq \init_w(A)$.~If equality holds,~$\{f_1, \dots, f_s\}$ is called a \textit{Khovanskii basis}~for~$A$. 

\begin{remark}\label{rmk: khovanskii bases more generally}
Khovanskii bases are defined in greater generality for certain rings equipped with discrete valuations \cite{kaveh2019khovanskii}. In this case the ring $A$ is equipped with the natural valuation induced by the weight $w$ \cite{clarke2022partialflag}. We note that, in the literature, Khovanskii bases have been known by many different names including SAGBI bases, canonical bases and subalgebra bases.
\end{remark}

We define the ideal $I \subseteq S =  \CC[y_1, \dots, y_s]$ to be the kernel of the map $S \rightarrow R$ where $y_i \mapsto f_i$. For each weight vector $w \in \RR^n$ we define the \textit{induced weight vector $\widehat w \in \RR^s$} for the ring $S$ given by $\widehat w = (w(f_1), w(f_2), \dots, w(f_s))$. The following result gives the connection between Khovanskii bases and toric degenerations.

\begin{theorem}[{\cite[Theorem~11.4]{sturmfels1996grobner}}]\label{thm: Khovanskii basis given by toric degeneration}
Fix $f_1, \dots, f_s \in R$. Assume that $w$ is \textit{generic}, i.e.~each initial form $\init_w(f_i)$ is a monomial. Then $f_1, \dots, f_s$ is a Khovanskii basis for the algebra they generate if and only if $\init_{\widehat w}(I) = \ker(S \rightarrow R \colon y_i \mapsto \init_w(f_i))$.
\end{theorem}

We combine the above theorem with the definition of a coherent matching field for the Grassmannian to obtain the following.

\begin{corollary}\label{cor: toric dege via matching fields}
Let $\Lambda$ be a coherent matching field induced by the weight $M \in \RR^{k \times n}$. The Pl\"ucker forms $\det(X_I)$ form a Khovanskii basis for the Pl\"ucker algebra if and only if the matching field gives rise to a toric degeneration $\init_{\widehat M}(G_{k,n}) = \ker(\CC[P_I] \rightarrow \CC[x_{i,j}] \colon P_I \mapsto \init_M(\det(X_I))$.
\end{corollary}

For any coherent matching field $\Lambda$, the kernel of the monomial map in the above corollary is called the \textit{matching field ideal}, denoted $I_\Lambda$. By \cite[Theorem~11.3]{sturmfels1996grobner}, it is always the case that $\init_{\widehat M}(G_{k,n}) \subseteq I_\Lambda$. 

\begin{remark}\label{rmk: NO-bodies}
Following \cite[Section~2.4]{clarke2022partialflag}, we note that the polytopes of matching fields that give rise to toric degenerations are Newton-Okounkov bodies for the Grassmannian. The valuation on the Pl\"ucker algebra is derived from the weight $w$ inducing the matching field.
\end{remark}

The matching field polytopes $P_\Lambda$ are normal \cite[Proposition~3]{clarke2022partialflag}, and we have the following.

\begin{theorem}[{\cite[Theorem~1]{clarke2022partialflag}}]\label{thm: toric degenerations if mutation equivalent}
Let $\Lambda$ be a coherent matching field for the Grassmannian. If the matching field polytope $P_\Lambda$ is mutation equivalent to the Gelfand-Tsetlin polytope, then $\Lambda$ gives rise to a toric degeneration of the Grassmannian. 
\end{theorem}

In the following section, we construct the FFLV-polytope for the Grassmannian as a matching field polytope. We will show that this polytope is mutation equivalent to the GT-polytope by a sequence of mutations that passes only through matching field polytopes. So, by Theorem~\ref{thm: toric degenerations if mutation equivalent}, we obtain a family of toric degenerations of the Grassmannian. By Corollary~\ref{cor: toric dege via matching fields}, for each such toric degeneration, we obtain a weight vector such that the Pl\"ucker forms are a Khovanskii basis for the Pl\"ucker algebra.

\section{FFLV-polytopes for Grassmannians}\label{sec: FFLV polytopes for gr}

In this section, we will show that the order and chain polytopes for the Grassmannian poset are matching field polytopes. Moreover, we will relate them by a sequence of mutations which passes only through matching field polytopes. We begin by fixing out notation for the section and describing connections between the Grassmannian poset, and the Gelfand-Tsetlin and Feigin–Fourier–Littelmann–Vinberg polytopes, which are usually defined for flag varieties.

\medskip

Throughout, fix $k < n $. The Grassmannian $\Gr(k,n)$ admits a toric degeneration to the toric variety associated to the diagonal matching field \cite{MS05}. The corresponding toric ideal is a Hibi ideal. That is, the ideal is generated by binomials of the form $P_\sigma P_\tau - P_{\sigma \wedge \tau} P_{\sigma \vee \tau}$ where $\sigma$ and $\tau$ are incomparable elements of a certain distributive lattice $\ML$. By Birkhoff's Representation Theorem, the lattice $\ML$ is determined by its poset of join irreducible elements. For the Grassmannian, this poset is precisely $Q_{k,n}$, see Definition~\ref{ex: grassmannian poset}. Recall that the order polytope $\MO(Q_{k,n})$ is naturally unimodular equivalent to the polytope $P_\MD$ of the diagonal matching field, see Theorem~\ref{thm: Gr poset order polytope equals diagonal MF polytope}.

The diagonal matching field gives rise to a toric degeneration of the flag variety $Fl_n$ embedded in a product of Grassmannians \cite{MS05}. The resulting projective toric variety is naturally the toric variety associated to the diagonal matching field polytope for the flag variety \cite{clarke2022partialflag}. This polytope can also be described as a Gelfand-Tsetlin polytope $\GT(\lambda)$ where $\lambda = (0,1,2, \dots, n-1)$. The polytope $\GT(\lambda)$ is the order polytope of a marked poset \cite[Section~4.1]{ardilla2011GTandFFLVPolytopes}, see Figure~\ref{fig: poset gt polytope}. Observe that the order polytope of the Grassmannian poset $\MO(Q_{k,n})$ is equal to the GT-polytope $\GT(\lambda(k,n))$ where $\lambda(k,n) := (0, \dots, 0, 1, \dots, 1)$ is the vector that contains $k$ zeros and $(n-k)$ ones. We define this polytope to be the GT-polytope of the Grassmannian.

\begin{figure}
    \centering
    \includegraphics[scale = 0.8]{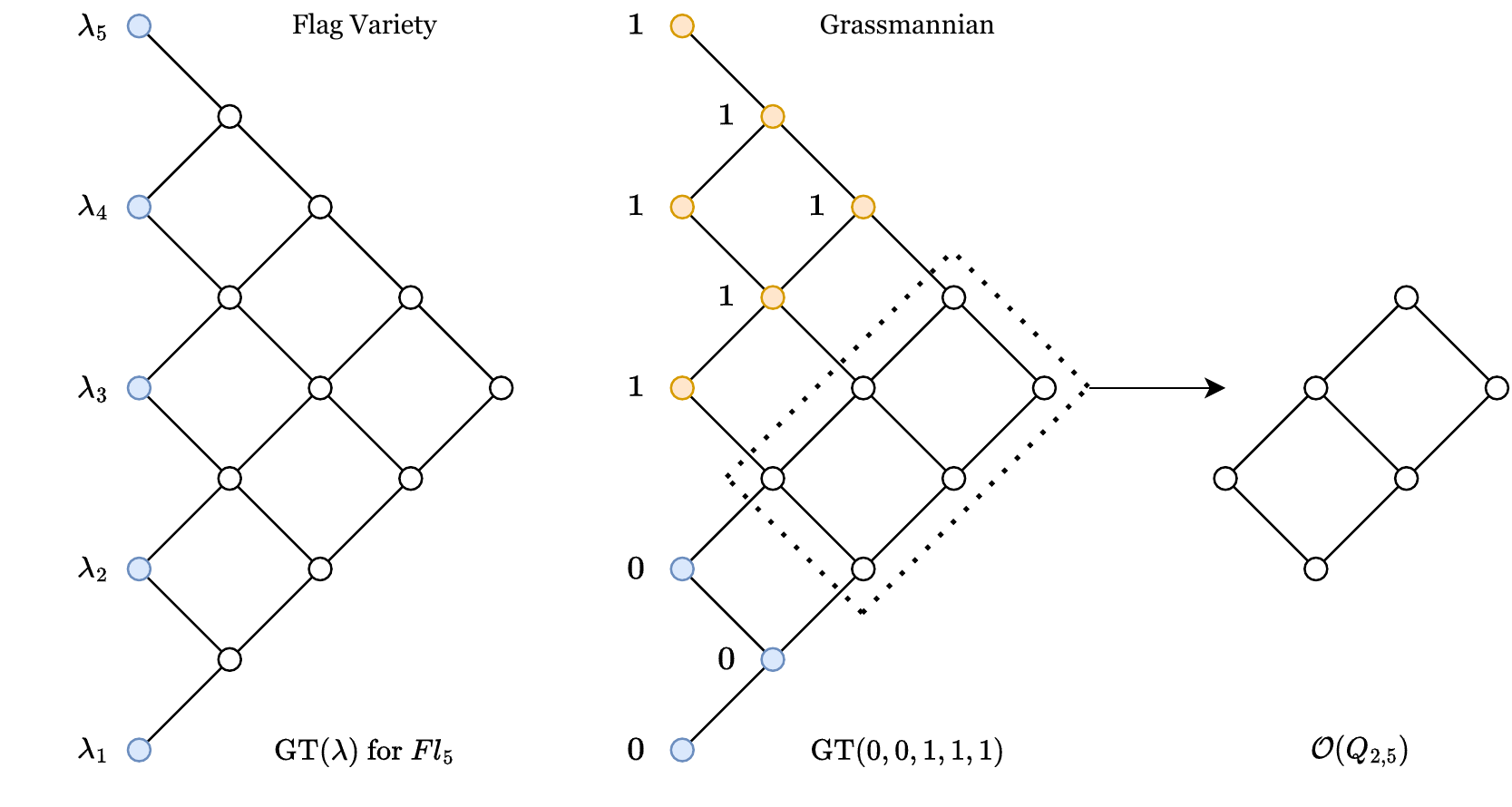}
    \caption{The Hasse diagrams for the posets associated to the flag variety (left) and the Grassmannian (right).}
    \label{fig: poset gt polytope}
\end{figure}

The Feigin–Fourier–Littelmann–Vinberg polytopes $\FFLV(\lambda)$ for the flag variety are given by the chain polytope of the same marked poset \cite{ardilla2011GTandFFLVPolytopes}. Following the above discussion of the GT-polytope for the Grassmannian, we naturally define the FFLV-polytope for the Grassmannian to be the chain polytope of the Grassmannian poset $\MC(Q_{k,n})$.

\subsection{The Grassmannian \texorpdfstring{$\Gr(3,n)$}{Gr(3,n)}}\label{sec: fflv for gr3n}

In this section, we show that the FFLV-polytope for the Grassmannian $\Gr(3,n)$ is given by the block diagonal matching field polytope in Example~\ref{ex: block diagonal matching field}.

\begin{theorem}\label{thm: gr3n fflv equals block diagonal}
Let $\MB$ be a block diagonal matching field.
The FFLV-polytope for $\Gr(3,n)$ is unimodular equivalent to the polytope $P_{\MB}$.
\end{theorem}

\begin{proof}
We begin by relabelling the elements of the poset $Q_{3,n}$ as $z_{i+2} := (1,i)$, $y_{i+2} := (2,i)$ and $x_{i+2} := (3,i)$ for each $i \in \{1,2,\dots,n-3\}$.  Note, with this notation, the elements of the poset are $\{x_i, y_i, z_i \colon 3 \le i \le n-1 \}$.
The anti-chains of $Q_{k,n}$ are the sets:
\begin{itemize}
    \item The empty set $\emptyset$,
    \item All singletons $\{x_i \}, \{y_i \}$ and $\{z_i \}$ with $3 \le i \le n-1$,
    \item Two-subsets $\{x_i, y_j \}$, $\{ x_i, z_j\}$ and $\{y_i, z_j \}$ where $3 \le i < j \le n-1$,
    \item Three-subsets $\{x_i, y_j, z_\ell \}$ where $3 \le i < j < \ell \le n-1$.
\end{itemize}

Recall that the tuples of the block diagonal matching field $\MB$
are $(i, 1, j)$ where $1 < i < j \le n$ and $(i, j, \ell)$ where $2 \le i < j < \ell < n$. We identify each triples with the corresponding vertex of the matching field polytope $\MP_{\MB} \subset \RR^{3 \times n}$. Explicitly, we fix a basis for $\RR^{3 \times n}$ given by $E = \{x_i, y_i, z_i : 1 \le i \le n\}$. A tuple $(i,j,\ell) \in [n]^3$ is identified with the corresponding point in $\RR^{3 \times n}$, which is $x_i + y_j + z_\ell$. We define the subspace $\RR^{3 \times (n-3)} \subseteq \RR^{3 \times n}$ as the span of $F = \{x_i, y_i, z_i : 3 \le i \le n-1 \}$.

We proceed to show that the FFLV-polytope and $P_\MB$ are unimodular equivalent by considering a projection of the matching field polytope. Consider the projection map
\[
\Pi : \RR^{3 \times n} \rightarrow \RR^{3 \times (n-3)}
\]
defined by its action on the basis $E$ by fixing $F$ and sending all elements in $E \backslash F$ to zero. The image of the vertices of $P_{\MB}$ under $\Pi$ are shown in Table~\ref{tab: vertices Gr3n under projection}.

\begin{table}
    \centering
    \begin{tabular}{lll}
        \toprule
        $v \in P_\MB$ & $\Pi(v)$ & Conditions  \\
        \midrule
        $(2,1,i)$ & $z_i$ & $3 \le i \le n-1$ \\
        $(2,1,n)$ & $0$ & \\
        $(i,1,j)$ &  $x_i + z_j$ & $3 \le i < j \le n-1$ \\
        $(i,1,n)$ &  $x_i$ & $3 \le i \le n-1$ \\
        $(2,i,j)$ &  $y_i + z_j$ & $3 \le i < j \le n-1$ \\
        $(2,i,n)$ &  $y_i$ & $3 \le i \le n-1$\\
        $(i,j,\ell)$ & $x_i + y_j + z_\ell$ & $3 \le i < j < \ell \le n-1$ \\
        $(i,j,n)$ & $x_i + y_j$ & $3 \le i < j \le n-1$ \\
        \bottomrule
    \end{tabular}
    \caption{Vertices of $P_\MB$ under the projection $\Pi$ in the proof of Theorem~\ref{thm: gr3n fflv equals block diagonal}}
    \label{tab: vertices Gr3n under projection}
\end{table}

The projection $\Pi$ sends the vertices of $\MP_\MB$ to the vertices of the FFLV-polytope, which are the points corresponding to the anti-chains of the poset. Therefore, the image of $\MP_{\MB}$ under $\Pi$ is the FFLV-polytope. Since the vertices of $\MP_{\MB}$ avoid the basis vectors $x_1, x_n, y_2, y_n, z_1$ and $z_2$ in $E$, the projection $\Pi$ gives rise to a unimodular equivalence of polytopes.
\end{proof}

Recall that the block diagonal matching field polytopes are mutation equivalent to the GT-polytope by Theorem~\ref{thm: block diagonal comb mutation equiv}. By Theorem~\ref{thm: gr3n fflv equals block diagonal}, the FFLV-polytope is unimodular equivalent to the block diagonal matching field polytope for $\Gr(3,n)$, which gives us the following result.

\begin{corollary}
The GT-polytope and FFLV-polytope for $\Gr(3,n)$ are related by a sequence of combinatorial mutations that pass only through matching field polytopes.
\end{corollary}

\subsection{FFLV matching field polytope for the Grassmannian \texorpdfstring{$\Gr(k,n)$}{Gr(k,n)}}\label{sec: fflv for grkn}

We now turn our attention to general Grassmannians $\Gr(k,n)$. A straightforward computation shows that Theorem~\ref{thm: gr3n fflv equals block diagonal} does not extend to $\Gr(4,8)$, i.e.~the chain polytope $\MC(Q_{k,n})$ is not unimodular equivalent to any block diagonal matching field polytope. Therefore, we define a new class of coherent matching fields which will give rise to the polytope $\MC(Q_{k,n})$.

\medskip

Throughout this section, we fix $k < n$. We define the $k \times n$ matrix $\rm Diag$ which gives rise to the diagonal matching field, see Example~\ref{ex: diagonal matching field}.~The entries of $\rm Diag$ are given by $({\rm Diag})_{i,j} = (i-1)(n+1-j)$. We give an explicit formulation of a matrix that induces the FFLV~matching~field.

\begin{definition}\label{def: fflv weight vector and matching field}
Let $N \in \mathbb{N}$ be a sufficiently large integer, for example $N = n^3$ is sufficient. Let $D = (d_{i,j})$ be the $k \times n$ matrix where $d_{i,j} = N$ if $i = j$ and $d_{i,j} = 0$ otherwise. We define
\[
M^{\rm FFLV} = {\rm Diag} - D.
\]

The matrix $M^{\rm FFLV}$ induces a matching field which we denote $\MB_{k,n}$. The requirement that `$N$ is sufficiently large' is equivalent to the condition that $\MB_{k,n}$ is well-defined. We call $\MB_{k,n}$ the \textit{FFLV matching field} for $\Gr(k,n)$.
\end{definition}

Before we state the result, let us first see an example of the matching field $\MB_{k,n}$.

\begin{example}\label{ex: FFLV matching field in gr37 case}
Let us consider the example for $\Gr(3,7)$. The diagonal matching field is induced by the matrix
\[
{\rm Diag} = 
\begin{bmatrix}
0 & 0 & 0 & 0 & 0 & 0 & 0 \\
7 & 6 & 5 & 4 & 3 & 2 & 1 \\
14 &12 &10 & 8 & 6 & 4 & 2
\end{bmatrix}.
\]
In this case we take $N = 20$, which is sufficiently large. We form the matrix $D$ and subtract it from $\rm Diag$ to obtain $M^{\rm FFLV}$
\[
\begin{bmatrix}
0 & 0 & 0 & 0 & 0 & 0 & 0 \\
7 & 6 & 5 & 4 & 3 & 2 & 1 \\
14 &12 &10 & 8 & 6 & 4 & 2
\end{bmatrix}
-
\begin{bmatrix}
  20 &  0 & 0& 0 & 0 & 0 & 0 \\
  0 & 20 & 0 & 0 & 0 & 0 & 0 \\
 0 &  0 & 20 & 0 & 0 & 0 & 0 
\end{bmatrix}
=
\begin{bmatrix}
 -20 & 0 &0& 0 & 0 & 0 & 0 \\
 7 &-14& 5 & 4 & 3 & 2 & 1 \\
14 & 12&-10 & 8 & 6 & 4 & 2
\end{bmatrix}.
\]

Suppose $S = \{s,t,u\} \subset [7]$ where $1 \le s < t < u \le 7$. The tableau representation of the tuples associated to $S$ with respect to $\MB_{3,6}$ is given by:
\[
\begin{tabular}{c c c c c c c c}
\begin{tabular}{|c|}
\hline
    $s$ \\
    $t$ \\
    $u$ \\
\hline
\end{tabular} &
\begin{tabular}{|c|}
\hline
    $1$ \\
    $t$ \\
    $u$ \\
\hline
\end{tabular} &
\begin{tabular}{|c|}
\hline
    $t$ \\
    $2$ \\
    $u$ \\
\hline
\end{tabular} &
\begin{tabular}{|c|}
\hline
    $t$ \\
    $u$ \\
    $3$ \\
\hline
\end{tabular} &
\begin{tabular}{|c|}
\hline
    $1$ \\
    $2$ \\
    $u$ \\
\hline
\end{tabular} &
\begin{tabular}{|c|}
\hline
    $1$ \\
    $u$ \\
    $3$ \\
\hline
\end{tabular} &
\begin{tabular}{|c|}
\hline
    $u$ \\
    $2$ \\
    $3$ \\
\hline
\end{tabular} & 
\begin{tabular}{|c|}
\hline
    $1$ \\
    $2$ \\
    $3$ \\
\hline
\end{tabular}
\\
if $s > 3$; &
if $s = 1$, &
if $s = 2$, &
if $s = 3$; &
if $s = 1$, &
if $s = 1$, &
if $s = 2$, &
if $s = 1$, \\
&
$t > 3$; &
$t > 3$; &
&
$t = 2$,&
$t = 3$;&
$t = 3$;&
$t = 2$,
\\
& & & &
$u > 3$; &
&
&
$u = 3$.
\end{tabular}
\]
For further details, see Section~\ref{sec: prelim matching fields and polytopes} or equivalently the image of the Pl\"ucker variable $P_S$ under the monomial map associated to $\MB_{3,6}$ in Corollary~\ref{cor: toric dege via matching fields}. In general, for $\Gr(k,n)$, we form the tuple associated to $S$ by putting each $i$ with $1 \le i \le k$ in the $i$-th position and arranging the remaining entries in increasing order. 
\end{example}

\begin{theorem} \label{thm: gr chain poset equals fflv polytope}
The FFLV-polytope for $\Gr(k,n)$ is unimodular equivalent to the matching field polytope $P_{\MB_{k,n}}$.
\end{theorem}
\begin{proof}
Let us consider the chain polytope $\MC(Q_{k,n})$ of $Q_{k,n}$ whose vertices correspond to anti-chains. This polytope naturally lives in $\RR^{k \times (n-k)}$, which has a basis in bijection with the elements $Q_{k,n}$. Let $Y = \{y_{i,j} : i \in [k], j \in [n-k] \}$ be a basis for $\RR^{k \times (n-k)}$ and define a bijection between $Y$ and $Q_{k,n}$ by
\[
(i,j) \longleftrightarrow y_{k+1-i, j}.
\]
We write the vertices of the chain polytope with respect to this bijection. So, the empty set corresponds to the zero vector; each singleton subset $\{(i,j)\}$ corresponds to the vector $y_{k+1-i, j}$; an anti-chain $\{ (i,j), (\ell,m)\}$ corresponds to the vector $y_{k+1-i,j} + y_{k+1-\ell, m}$; and so on. Observe that a pair of poset elements are incomparable if and only if the corresponding basis elements $y_{i,j}$ and $y_{i',j'}$ satisfy: $i < i'$ and $j < j'$, or $i' < i$ and $j' < j$.

Since the matching field polytope $P_{\MB_{k,n}}$ naturally lives in $\RR^{k\times n}$ with basis $\{e_{i,j} \colon i \in [k], j \in [n] \}$, we now define a projection
\[
\Pi : \RR^{k \times n} \rightarrow \RR^{k \times (n-k)}\quad\text{with}\quad e_{i,j} \mapsto
\left\{
\begin{tabular}{ll}
    $y_{i,j-k}$ & if $k+1 \le j \le n$, \\
    $0$ & otherwise.
\end{tabular}
\right.
\]
We show that $\Pi$ induces a unimodular equivalence between $P_{\MB_{k,n}}$ and $\MC(Q_{k,n})$. For each $k$-subset $I = \{i_1 < \dots < i_k \}$, consider the corresponding tuple of the matching field $\MB_{k,n}$. The tuple is uniquely determined by the following criteria:
\begin{itemize}
    \item If $i_j \in [k]$, then it appears in position $i_j$ in the tuple,
    \item The entries $i_j$ that do not lie in $[k]$, appear in increasing order in the tuple.
\end{itemize}
Under the map $\Pi$, the basis vectors $e_{i,j}$ where $j \in [k]$, which correspond to the entries of the tuple that lie in $[k]$, are sent to $0$. The remaining entries of the tableau are mapped by $e_{i,j} \mapsto y_{i,j-k}$ for $k+1 \le j \le n$. Since the entries of the tuple that are not killed by $\Pi$ appear in increasing order, it follows that the basis vectors that appear in $\Pi(v_I)$ correspond to pairwise incomparable elements of $Q_{k,n}$, i.e.~an anti-chain. So we have showed that each vertex $v_I$ is mapped to a vertex of the chain polytope. It remains to show that this map between vertices is injective.

By the above criteria for tuples of the matching field $\MB_{k,n}$, the subset $I$ is uniquely determined by the entries that lie in $\{k+1,\dots,n \}$ and their position in the tuple. This is precisely the information given by the image of a vertex under the map $\Pi$. Therefore, the map $\Pi$ restricts to a bijection between the vertices of $P_{\MB_{k,n}}$ and $\MC(Q_{k,n})$. Since $\Pi$ is a projection, it follows that these polytopes are unimodular equivalent.
\end{proof}

\subsection{Intermediate matching field polytopes}\label{sec: intermediate matching field polytopes}

In this section, we construct the intermediate polytopes that appears in the sequence of mutations from the GT-polytope to the FFLV-polytope in the proof of Theorem~\ref{thm:main_theorem}. We proceed by introducing a sequence of coherent matching fields that interpolates between the diagonal matching field and the FFLV matching field. As an outline of the construction: we first define a sequence of triples $S(k,n)$, we use these triples to define a sequence of weight matrices $M(k,n)$, and we then show that the weight matrices induce the desired matching fields $B(k,n)$.

\begin{remark} \label{rmk: two fflv weight vectors}
In the construction below, we give an alternative description of a weight matrix that induces the FFLV matching field. We note that the weight matrix $M^{\rm FFLV}$ simplifies the proof of Theorem~\ref{thm: gr chain poset equals fflv polytope} because the defining criteria for the tuples of the matching field immediately follow from $M^{\rm FFLV}$, see the proof of Theorem~\ref{thm: gr chain poset equals fflv polytope} and Example~\ref{ex: FFLV matching field in gr37 case}. However, the definition below is amenable to the construction of combinatorial mutations because the intermediate matching fields are defined inductively by permuting the entries of the weight matrices.
\end{remark}

Throughout this section, fix $k < n$. Recall $\MD$ the diagonal matching field for $\Gr(k,n)$ induced by the generic weight matrix $M_D$, see Example~\ref{ex: diagonal matching field}.
Observe that $M_D$ has the property that, for any $i \in [k-1]$, the difference between any pair of entries in row $i$ is less than the difference between any two distinct entries in row $i+1$. Therefore, any row-wise permutation of the matrix $M_D$ induces a coherent matching field. 

\medskip

We inductively define a finite sequence $S(k,n)$ of triples $(p_i, \ell_i, q_i) \in [n]^3$ with $i \ge 0$ as follows. We define $(p_0, \ell_0, q_0) = (k+1, k, n)$. Given $(p_i, \ell_i, q_i)$ for some $i \ge 0$, then $(p_{i+1}, \ell_{i+1}, q_{i+1})$ is given by:
\begin{itemize}
    \item $(p_i, \ell_i, q_i + 1)$ if $q_i < n$,
    \item $(p_i, \ell_i + 1, k+1)$ if $q_i = n$ and $\ell_i < p_i - 1$,
    \item $(p_i - 1, 1, k+1)$ if $p_i > 2$, $\ell_i = p_i - 1$ and $q_i = n$.
\end{itemize}
The sequence terminates at the triple $(2, 1, n)$.

\smallskip

For example if $k = 3$ and $n = 6$ then the sequence of triples is given by 
\[
(436, 314, 315, 316, 324, 325, 326, 214, 215, 216).
\]

\medskip

We inductively define a finite sequence $M(k,n) = (M_0, M_1, \dots)$ of weight matrices with one weight matrix for each entry of $S(k,n)$. The first weight matrix in the sequence is $M_0 = M_D$. See Example~\ref{ex: diagonal matching field}. Assume that we have defined the weight matrix $M_{i-1}$ for some $i \ge 1$. Then the weight matrix $M_i$ is obtained from $M_{i-1}$ by swapping the entries $(M_{i-1})_{\ell_i+1, p_i}$ and $(M_{i-1})_{\ell_i + 1, q_i}$.

\begin{example} \label{example: gr36 intermediate weight matrices}
For $(k,n) = (3,6)$ the first few weight matrices in the sequence $M(3,6)$ are:
\[
M_D = \begin{bmatrix}
0 & 0 & 0 & 0 & 0 & 0 \\
5 & 4 & 3 & 2 & 1 & 0 \\
30 & 24 & 18 & 12 & 6 & 0
\end{bmatrix},
\begin{bmatrix}
0 & 0 & 0 & 0 & 0 & 0 \\
5 & 4 & \boxed{2} & \boxed{3} & 1 & 0 \\
30 & 24 & 18 & 12 & 6 & 0
\end{bmatrix},
\begin{bmatrix}
0 & 0 & 0 & 0 & 0 & 0 \\
5 & 4 & \boxed{1} & 3 & \boxed{2} & 0 \\
30 & 24 & 18 & 12 & 6 & 0
\end{bmatrix},
\]
\[
\begin{bmatrix}
0 & 0 & 0 & 0 & 0 & 0 \\
5 & 4 & \boxed{0} & 3 & 2 & \boxed{1} \\
30 & 24 & 18 & 12 & 6 & 0
\end{bmatrix},
\begin{bmatrix}
0 & 0 & 0 & 0 & 0 & 0 \\
5 & 4 & 0 & 3 & 2 & 1 \\
30 & 24 & \boxed{12} & \boxed{18} & 6 & 0
\end{bmatrix},
\begin{bmatrix}
0 & 0 & 0 & 0 & 0 & 0 \\
5 & 4 & 0 & 3 & 2 & 1 \\
30 & 24 & \boxed{6} & 18 & \boxed{12} & 0
\end{bmatrix},\dots
\]
where the swapped entries are shown in boxes.
\end{example}

Each weight matrix $M_i$ induces a coherent matching field, which we write as $B(k,n)_i$. The polytope corresponding to the matching field $B(k,n)_i$ is denoted $P(k,n)_i \subseteq \RR^{k \times n}$. The tuples that define the matching field can be inductively described as follows:

\begin{proposition}\label{prop gr fflv matching field inductive description}
Fix $k < n$ and $i \in \{1, \dots, \#S(k,n) \}$. Then the matching field
$B(k,n)_i$ is given by the tuples
\begin{align*}
    B(k,n)_i &= \{(j_1, \dots, j_k) \in B(k,n)_{i-1} \colon j_\ell \neq p_i \text{ or } j_{\ell+1} \neq q_i \}\\
    & \cup \{(j_1, \dots, j_{\ell-1}, q_i, p_i, j_{\ell+1}, \dots, j_k) \colon (j_1, \dots, j_k ) \in B(k,n)_{i-1}, \ j_\ell = p_i \text{ and } j_{\ell+1} = q_i \}.
\end{align*}
In other words, the tuples of $B(k,n)_i$ are obtained from $B(k,n)_{i-1}$ by swapping the positions of $p_i$ and $q_i$ if they appear in positions $\ell$ and $\ell+1$.
\end{proposition}

\begin{proof}
Let $M$ be any matrix whose entries are row-wise equal to $M_D$. Let $J  = \{j_1 < \dots < j_k\} \subseteq [n]$ be a $k$-subset and consider the Pl\"ucker form $P_J \in \mathbb{C}[x_{s,t}]_{s \in [k], t \in [n]}$. Suppose that the initial term of $P_J$ with respect to $M$ is given by $c\prod_{s = 1}^k x_{s, j_s}$ for some $c \in \{+1, -1\}$. 

\medskip \noindent {\bf Claim 1.}
\emph{For each $s \in [k]$, we have that $M_{s,j_s} = \min \{M_{1, j_1}, M_{2, j_2}, \dots, M_{s, j_{s}}\}$.}

\begin{proof}
We proceed by induction on $k$. If $k = 1$ then the claim holds trivially. For the inductive step, it suffices to show the claim for $s = k$. Suppose that $M_{s', j_{s'}} =  \min\{M_{1, j_1}, M_{2, j_2}, \dots, M_{k, j_{k}}\}$ for some $s' \neq k$. By the definition of $M_D$, we have $M_{k, j_k} - M_{s', j_{s'}} \ge n^{k - 2}$. Let $\{j_1', \dots, j_{k-1}'\} = \{j_1, \dots, j_k \} \backslash j_{s'}$. By the definition of $M_D$, we have that $M_{t, j_t'} \le (n-1)n^{t-2}$ for each $t \in \{2, \dots, k-1\}$ and $M_{1,j_1'} = 0$. We have that $c'x_{s',j_{s'}} \prod_{t = 1}^{k-1}x_{t, j_t'}$ is a term of $P_J$ for some $c' \in \{+1, -1\}$. The weight of this term with respect to $M$ is given by
\[
M_{s', j_{s'}} + \sum_{t = 1}^{k-1} M_{t, j_t'} \le M_{s, j_s} + \sum_{t = 2}^{k-1} (n-1)n^{t-2} = M_{s, j_s} + n^{k-2} - (n-1) < M_{k, j_k}.
\]
Therefore, the initial term of $P_J$ is not $c\prod_{s = 1}^k x_{s, j_s}$, a contradiction, and we have shown the claim.
\end{proof}

Let $(j_1, \dots, j_k) \in B(k,n)_{i-1}$ be a tuple. Note that the only difference between the matrices $M(k,n)_{i-1}$ and $M(k,n)_i$ is in row $\ell_i + 1$. So, by Claim~1, we have that the tuples corresponding to $(j_1, \dots, j_k)$ in $B(k,n)_i$ and in $B(k,n)_{i-1}$ coincide from position $\ell_i+2$ to position $k$. 

\medskip \noindent {\bf Claim 2.}
\emph{We have that $(M(k,n)_{i-1})_{\ell_i + 1, p_i} 
- (M(k,n)_{i-1})_{\ell_i + 1, q_i} 
= n^{\ell_i - 1}
$, in particular this difference is as small as possible in row $\ell_i + 1$ of $M(k,n)_{i-1}$.}

\begin{proof}
We imagine that the matrix $M(k,n)_{i-1}$ is obtained from $M_D$ by a sequence of swaps among its entries. Note that it suffices to consider only the swaps that occur in row $\ell_i +1$. Each such swap is associated to an element of $S(k,n)_j$ where $\ell_j = \ell_i$ and the subsequence of such elements is given by
\[
(k, \ell_i, k+1), (k, \ell_i, k+2), \dots, (k, \ell_i, n), (k-1, \ell_i, k+1), \dots, (k-1, \ell_i, n), \dots, (\ell_i+1, \ell_i, n).
\]
For ease of notation, we write down the entries of row $\ell_i+1$ of the weight matrices divided by $n^{\ell_i-1}$, so for $M_D$ this gives $(n-1, n-2, \dots, 1, 0)$. Consider the effect of applying the swaps defined by $(k, \ell_i, k+1), \dots, (k, \ell_i, n)$ to $M_D$. Each swap interchanges the entry in column $k$ with the next smallest entry in that row. As a result, the second row of $M(k,n)_j$ where $S(k,n)_j = (k,\ell_i, n)$ is given by $(n-1, n-2, \dots, n-k+1, 0, n-k, n-k-1, \dots, 2,1)$. Similarly, for each $s \in \{ 1, \dots, k-\ell_i\}$, the second row of $M(k,n)_j$ where $S(k,n)_j = (k-s+1, \ell_i, n)$ is given by $(n-1, n-2, \dots, n-k+s, s-1, s-2, \dots, 0, n-k+s-1, n-k+s-2, \dots, s)$, and each swap interchanges an element with the next smallest entry in the same row. This concludes the proof of the claim.
\end{proof}

By Claim~2, we have that $(j_1, \dots, j_k)$ is a tuple of $B(k,n)_i$ if either $j_{\ell_i+1} \neq p_i$ or $j_{\ell_i+1} \neq q_i$. On the other hand, if $j_{\ell_i} = p_i$ and $j_{\ell_i+1} = q_i$ then we have that the corresponding entries of $M(k,n)_i$ are swapped. And so the corresponding tuple is obtained by swapping the position of $p_i$ and $q_i$ in the tuple. This concludes the proof of the lemma. 
\end{proof}

\noindent 
\textbf{Notation.} Whenever $k$ and $n$ are fixed, we omit them from the notation of $S(k,n)$, $M(k,n)$, $B(k,n)$ and $P(k,n)$ and write $S$, $M$, $B$ and $P$, respectively.

\medskip

We give an explicit description of the tuples that define $B_i$ as follows.

\begin{lemma}\label{lemma gr fflv vertices}
Let $J = \{j_1 < \dots < j_k \} \subseteq [n]$ be a $k$-subset. Construct the tuple $T(i)_J = (t_1, \dots, t_k)$ whose entries are the elements $J$ satisfying the following conditions:
\begin{itemize}
    \item If $j \in J \cap \{p_i+1, \dots, k \}$ then $t_j = j$,
    \item If $j_s = p_i \in J$ for some $s \le \ell_i$ and $j_{s+d} \le q_i$ where $d \ge 1$ is the smallest value such that $j_{s+d} \ge k+1$ then $t_{\ell_i + 1} = j_s$,
    \item If $j_s = p_i \in J$ for some $s \le \ell_i$ and $j_{s+1} > q_i$ where $d \ge 1$ is the smallest value such that $j_{s+d} \ge k+1$ then $t_{\ell_i} = j_s$,
    \item The entries of $T(i)_J$ not listed above are in ascending order.
\end{itemize}
Then the tuples of $B_i$ are exactly $T(i) := \left\{T(i)_J \colon J \in \binom{[n]}{k} \right\}$.
\end{lemma}

\begin{proof}
For ease of notation, we use the term \textit{conditions} to refer to the conditions satisfied by the tuple in the statement of the lemma. The proof follows by induction on $i \in \{1, \dots, \#S\}$. Suppose $i = 1$ and fix a $k$-subset $J \subseteq[n]$. We have $S_i = (k, 1, k+1)$ so the set $\{p_i+1, \dots, k \}$ is empty and there are no elements of $T(i)_J$ determined by the first condition. Consider the second and the third conditions. If we have $j_s = p_i$ and $s \le \ell_i$ then it follows that $j_1 = k$. If $j_2 \le q_i$ then $j_2 = k+1$ and so $t_2 = j_1 = k$. In this case, the remaining entries of $T(i)_J$ are listed in ascending order so we have $T(i)_J = (j_2, j_1, j_3, \dots, j_k) = (k+1, k, j_3, \dots, j_k)$. Otherwise if $j_2 > q_i$ then it follows that $j_2 > k+1$ and $T(i)_J = (j_1, j_2, \dots, j_k)$ is given in ascending order. By Proposition~\ref{prop gr fflv matching field inductive description}, we have that $T(1)$ coincides with the tuples of $B_1$.

For the inductive step, fix some $i \ge 1$ and assume that the tuples of $B_i$ coincide with $T(i)$. By Proposition~\ref{prop gr fflv matching field inductive description}, the matching field $B_{i+1}$ is obtained by swapping the positions of $p_{i+1}$ and $q_{i+1}$ within all tuples of $B_i$ for which $p_{i+1}$ appears in position $\ell_{i+1}$ and $q_{i+1}$ in position $\ell_{i+1}+1$. Let $J = \{j_1 < \dots < j_k \}$ be a $k$-subset of $[n]$. Suppose that $T(i)_J \neq T(i+1)_J$. We proceed by taking cases on $p_i$, $\ell_i$ and $q_i$ which determine $S_{i+1}$.

\smallskip

\noindent {\bf Case 1.}
Assume that $q_i < n$ and so $S_{i+1} = (p_i, \ell_i, q_i + 1)$. It follows that $j_s = p_i$ for some $s \le \ell_i$. Let $d \ge 1$ be the smallest value such that $j_{s+d} \notin \{ p_i+1, \dots, k\}$. Since $T(i)_J \neq T(i+1)_J$, it follows that $j_{s+r} = q_i + 1$. By the inductive hypothesis, we have that $T(i)_J$ is a tuple of $B_i$. Within $T(i)_J$, we have that $p_i$ appears in position $\ell_i$. By an easy inductive argument using the definition of the matching fields $B_a$ where $1 \le a \le i$, it follows that $j_{s+r}$ appears in position  $\ell_i+1$ in $T(i)_J$. In the tuple $T(i+1)_J$ we have that the positions of $p_i$ and $q_i+1$ are swapped, which concludes this case.

\smallskip

\noindent {\bf Case 2.}
Assume that $q_i = n$ and $\ell_i < p_i - 1$ and so $S_{i+1} = (p_i, \ell_i + 1, k+1)$. It follows that $j_s = p_i$ for some $s \le \ell_i$. Let $d \ge 1$ be the smallest value such that $j_{s+d} \notin \{p_i+1, \dots, k \}$. Since $T(i)_J \neq T(i+1)_J$, it follows that $j_{s+d} = k+1$. By a simple inductive argument, it is easy to see that the position of $p_i$ in $T(i)_J$ is $\ell_i+1$. Therefore, the position of $k+1$ in $T(i)_J$ is $\ell_i + 2$. On the other hand, the position of $p_i$ in $T(i+1)_J$ is $\ell_i+2$ and the position of $k+1$ is $\ell_i+1$. Since these are the only entries of the tuples which are different, we are done with this case.

\smallskip

\noindent {\bf Case 3.}
Assume that $q_i = n$ and $\ell_i = p_i - 1$ and so $S_{i+1} = (p_i - 1, 1, k+1)$ where $p_i > 2$. It follows that $j_1 = p_i - 1$. Let $d \ge 1$ be the smallest value such that $j_d \notin \{p_i, \dots, k\}$. Since $T(i)_J \neq T(i+1)_J$ we have that $j_d = k+1$. We see that $p_i-1$ appears in position $1$ in $T(i)_J$ and, by a straightforward inductive argument, that $j_d$ appears in position $2$. In $T(i+1)_J$, we have that $p_i - 1$ appears in position $2$ and $j_d$ appears in position $1$. Since these are the only entries of the tuples which are different, we are done with this case. 

\smallskip

In each case we have shown that the tuples in $T(i)$ which do not appear in $T(i+1)$ are exactly those which contain $p_i$ in position $\ell_i$ and $q_i$ in position $\ell_i+1$. We have also shown that swapping these two entries gives a tuple in $T(i+1)$ and so we are done.
\end{proof}

Let $s = \#S-1$ be the index of the last entry in the sequence $S$. In the following result, we show that the matching fields $\MB_{k,n}$ (see Definition~\ref{def: fflv weight vector and matching field}) and $B_s$, induced by the weight matrices $M^{\rm FFLV}$ and $M_s$ respectively, are the same. 

\begin{proposition}\label{prop same fflv matching field for different weights}
Let $s = \#S-1$. The matching fields induced by the weight matrices $M^{\rm FFLV}$ and $M_s$ are equal. In particular, their induced weight vectors lie in the same top-dimensional cone of the tropical Grassmannian.
\end{proposition}

\begin{proof}
The explicit description of the matching field associated to $M_s$ is given in Lemma~\ref{lemma gr fflv vertices}. In particular, for any subset $J = \{j_1 < \dots < j_k \}$, the corresponding tuple is given by $T(s)_J = (t_1, \dots, t_k)$ where $t_j = j$ for each $j \in \{1, \dots, k\}$ and the remaining elements of $T(s)_J$ are in ascending order. This tuple is identical to the tuple induced by the weight matrix $M^{\rm FFLV }$. 
\end{proof}

\subsection{Mutation equivalence of FFLV and GT matching field polytopes}\label{sec: fflv and gt mutation equiv}

In this section, we show that the FFLV and GT matching field polytopes are mutation equivalent. 

\begin{theorem}\label{thm:main_theorem}
The GT-polytope and the FFLV-polytope for $\Gr(k,n)$ are connected by a sequence of combinatorial mutations.
\end{theorem}

We recall from Section~\ref{sec: intermediate matching field polytopes} the sequence of triples $S = S(k,n)$, the weight matrices $M = M(k,n)$ and matching fields $B = B(k,n)$. For each $i \in \{1, \dots, \#S\}$, we have the triple $S_i = (p_i, q_i, \ell_i) \in [n]^3$, the weight matrix $M_i \in \RR^{k \times n}$ and the matching field $B_i$. To prove Theorem~\ref{thm:main_theorem}, we construct a combinatorial mutation which takes the polytope $P_{i-1}$ associated to $B_{i-1}$ to the polytope $P_{i}$ associated to $B_i$ for each $i \in \{1, \dots, \#S\}$. We begin by defining a tropical map.

\begin{definition}\label{fflv gr tropical map}
Fix $k < n$ and $i \in \{ 1, \dots, \#S\}$. For the following definitions we recall that $S_i = (p_i, \ell_i, q_i) \in [n]^3$ is a triple of natural numbers. We define $w^i \in \RR^{k \times n}$ to be the matrix given by
\[
w^i_{s,t} =
\begin{cases}
1 & \text{if } (s, t) \in \{(\ell_i, q_i), (\ell_i + 1, p_i) \}, \\
-1 & \text{if } (s, t) \in \{(\ell_i, p_i), (\ell_i + 1, q_i) \}, \\
0 & \text{otherwise.}
\end{cases}
\]
We define the matrix $f^i \in \RR^{k \times n}$ by
\[
f^i_{s,t} = 
\begin{cases}
-1 & \text{if } s = \ell_i \text{ and } t \in \{p_i, q_i, q_i+1, \dots, n\}, \\
 1 & \text{if } s = \ell_i + 1 \text{ and } t \ge q_i + 1,\\
 0 & \text{otherwise.}
\end{cases}
\]
It is convenient to view $w^i$ and $f^i$ as matrices 
\begin{gather*}
w^i = 
\bordermatrix{
         &              &p_i&             &q_i&             \cr
1        & 0\  \cdots\ 0 & 0 & 0\ \cdots\ 0 & 0 & 0\ \cdots\ 0 \cr
\vdots   & \ddots  & 0 & \ddots & 0 & \ddots \cr
\ell_i   & 0\  \cdots\ 0 &-1 & 0\ \cdots\ 0 & 1 & 0\ \cdots\ 0 \cr
\ell_i+1 & 0\  \cdots\ 0 & 1 & 0\ \cdots\ 0 &-1 & 0\ \cdots\ 0 \cr
\vdots   & \ddots  & 0 & \ddots & 0 & \ddots \cr
k        & 0\  \cdots\ 0 & 0 & 0\ \cdots\ 0 & 0 & 0\ \cdots\ 0
} 
\ \  \text{and} \ \ 
f^i = 
\bordermatrix{
         &              &p_i&             &q_i&             \cr
1        & 0\  \cdots\ 0 & 0 & 0\ \cdots\ 0 & 0 & 0\ \cdots\ 0 \cr
\vdots   & \ddots  & 0 & \ddots & 0 & \ddots \cr
\ell_i   & 0\  \cdots\ 0 &-1 & 0\ \cdots\ 0 &-1 &-1\ \cdots\ {-1} \cr
\ell_i+1 & 0\  \cdots\ 0 & 0 & 0\ \cdots\ 0 & 0 & 1\ \cdots\ 1 \cr
\vdots   & \ddots  & 0 & \ddots & 0 & \ddots \cr
k        & 0\  \cdots\ 0 & 0 & 0\ \cdots\ 0 & 0 & 0\ \cdots\ 0
}.
\end{gather*}

By the above depiction, it is clear that $f^i \in (w^i)^\perp$. So, we define the tropical map $\varphi^i := \varphi_{w^i, F}$ where $F = \conv\{\underline 0, f^i\}$.
\end{definition}

We show that the tropical map $\varphi^i$ defines a combinatorial mutation between the polytopes $P_{i-1}$ and $P_i$ for each $i \in \{1, \dots, \#S \}$. In the following sections, unless otherwise stated, we will fix $k < n$ and $i \in \{1, \dots, \#S \}$.

\begin{lemma}\label{lemma gr fflv inner products}
Let $J = \{j_1 < \dots < j_k \} \subseteq [n]$ be a $k$-subset. Let $j_s \in J$ denote the element that appears in position $\ell_i$ in the tuple corresponding to $J$ in $B_i$. Let $d \ge 1$ be the smallest value such that $j_{s + d} \ge k+1$. The inner product of $f^i$ with the vertex $v^{i-1}_J \in V(P_{i-1})$ is
\[
\langle f^i, v^{i-1}_J \rangle = 
\begin{cases}
-1 & \text{if } j_s = p_i \text{ and } j_{s + d} = q_i, \\
1 & \text{if } p_i < j_s < q_i < j_{s + d} \text{  or } j_s < p_i < q_i < j_{s + d},\\
0 & \text{otherwise.}
\end{cases}
\]
The inner product of $f^i$ with the vertex $v^i_J \in V(P_i)$ is
\[
\langle f^i, v^i_J \rangle = 
\begin{cases}
-1 & \text{if } j_s = q_i \text{ and } j_{s + d} = p_i, \\
1 & \text{if } p_i < j_s < q_i < j_{s + d} \text{  or } j_s < p_i < q_i < j_{s + d},\\
0 & \text{otherwise.}
\end{cases}
\]
\end{lemma}

\begin{proof}
First consider the polytope $P_{i-1}$. By Lemma~\ref{lemma gr fflv vertices}, we have that $j_{s+d}$ appears in position $\ell_i+1$ in the tuple associated to $J$. Suppose that $j_s = p_i$. If $j_{s+d} = q_i$ then the inner product of $v^{i-1}_J$ with $f^i$ is $-1$. If $j_{s+d} \neq q_i$, then by Lemma~\ref{lemma gr fflv vertices}, it follows that $j_{s+d} > q_i$ and so the inner product of $v^{i-1}_J$ with $f^i$ is $0$. Now suppose that $j_s \neq p_i$. Since $j_s < j_{s+d}$, the inner products in the statement of the lemma follow immediately.

Next consider the polytope $P_i$. By definition, the tuples of $B_i$ are obtained from $B_{i-1}$ by swapping $p_i$ and $q_i$ if they appear in positions $\ell_i$ and $\ell_i+1$ respectively. However this does not affect the inner product of such tuples with $f^i$. And so the analogous result holds for $P_i$.
\end{proof}

\begin{lemma}\label{lemma gr fflv no bad edges}
Suppose that $v_I, v_J \in V(P_i)$ are vertices such that $\langle f^i, v_I \rangle = 1$ and $\langle f^i, v_J\rangle = -1$. Then there exist vertices $v_{I'}, v_{J'} \in V(P_i) \cap (f^i)^\perp$ such that $v_I + v_J = v_{I'} + v_{J'}$. The same also holds for $P_{i-1}$.
\end{lemma}

\begin{proof}
We first focus on the polytope $P_i$. Since $i$ is fixed and we will not consider the polytope $P_{i-1}$ until we focus on the other case, we will omit $i$ from the notation of the polytope $P_i$, vector $f^i$, and natural number $p_i, \ell_i$ and $q_i$ to avoid confusion with certain elements of subsets. We will write $T(I)$ for the tuple $T(i)_I$ associated to the set $I$. We write $I = \{i_1 < \dots < i_k \}$ and $J = \{j_1 < \dots < j_k\}$ for the two $k$-subsets of $[n]$.

Since $\langle f, v_J\rangle = -1$, by Lemma~\ref{lemma gr fflv inner products}, we have that $T(J)_\ell = q$, that is the entry of $T(J)$ in position $\ell$ is $q$ and $T(J)_{\ell+1} = p$. Suppose that $T(I)_\ell = i_s$ for some $s$. By Lemma~\ref{lemma gr fflv vertices}, we have that $T(I)_{\ell+1} = i_{s+d}$ where $d \ge 1$ is the smallest value such that $i_{s+d} \ge k+1$. Since $\langle f, v_I \rangle = 1$, by Lemma~\ref{lemma gr fflv inner products}, we have that $i_{s+d} > q$ and either: $i_s < p$; or $p < i_s < q$. We define the sets
\begin{eqnarray}\label{eq:def:I'}
I' = \{T(I)_1, \dots, T(I)_\ell, T(J)_{\ell+1}, \dots, T(J)_k\} 
\text{ and }
J' = \{T(J)_1, \dots, T(J)_\ell, T(I)_{\ell+1}, \dots, T(I)_k\}.
\end{eqnarray}
We will show that the entries above within each set are distinct and that the order of the entries in $T(I')$ and $T(J')$ coincide with the order above.

Note that for any $i \in I \cap \{p+1, \dots, k\} $ we have that $T(I)_i = i$ and similarly for $J$. Since $p < \ell$, we have that none of $T(I)_1, \dots, T(I)_\ell$ or $T(J)_1, \dots, T(J)_\ell$ lie in $\{p+1, \dots, k\}$. So, without loss of generality, we may assume that $I \cap \{p+1, \dots, k\} = \emptyset$ and $J \cap \{p+1, \dots, k \} = \emptyset$. Since $p \notin I$ we have that $T(I)$ is in ascending order and all entries of $T(J)$ except $p$ are in ascending order.

Consider the set $J'$. Since $T(J)_\ell = q < i_{s+d} = T(I)_{\ell+1}$ it follows that the entries of $J'$ are distinct. Moreover the entries of tuple $T(J')$ are exactly in the order shown in the definition~of~$J'$.

Next consider the set $I'$. If $i_s < p$ then it follows that $T(I)_\ell = i_s < p = T(J)_{\ell+1} < T(J)_{\ell+2}$ and so all the entries in $I'$ are distinct and $T(I')$ is in the order given in \eqref{eq:def:I'} in the definition of $I'$. On the other hand, if $p < i_s < q$ then note that $T(I)_{\ell} = i_s < q = T(J)_\ell < T(J)_{\ell+2}$. Since $p$ does not lie in $I$, it follows that all entries of $I'$ are distinct. Moreover, $i_s < q$, so by Lemma~\ref{lemma gr fflv vertices}, it follows that the entries of $T(I)$ are ordered as in the definition of $I'$. As a result we have shown that $v_I + v_J = v_{I'} + v_{J'}$. Since $T(J')_{\ell} = q$ and $T(J')_{\ell} > q$, we see that $\langle f, v_{J'}\rangle = 0$ and so it follows that $\langle f, v_{J'} \rangle = 0$ and we are done for the polytope $P_i$.

\medskip

We now consider the polytope $P_{i-1}$. For this case, we will omit the $i-1$ from the polytope $P_{i-1}$ and write $T(I)$ for the tuple $T(i-1)_I$. Note that, as in the previous case, we write $f$ for the vector $f^i$ and omit the $i$ from $p_i$, $\ell_i$ and $q_i$.

Since $\langle f, v_J \rangle = -1$, we have that $T(J)_{\ell} = p$ and $T(J)_{\ell+1} = q$. Suppose that $T(I)_\ell = i_s$ for some $s$. By Lemma~\ref{lemma gr fflv vertices}, we have that $T(I)_{\ell+1} = i_{s+d}$ where $d \ge 1$ is the smallest value such that $i_{s+d} \ge k+1$. Since $\langle f, v_I \rangle = 1$, by Lemma~\ref{lemma gr fflv inner products}, we have that $i_{s+d} > q$ and either: $i_s < p$; or $p < i_s < q$. We define the sets
\begin{eqnarray}\label{eq:def:I'fori-1}
I' = \{T(I)_1, \dots, T(I)_\ell, T(J)_{\ell+1}, \dots, T(J)_k\} 
\text{ and }
J' = \{T(J)_1, \dots, T(J)_\ell, T(I)_{\ell+1}, \dots, T(I)_k\}.
\end{eqnarray}
We will now show that the entries above within each set are distinct and that the order of the entries in $T(I')$ and $T(J')$ coincide with the order above.

Note that for any $i \in I \cap \{p+1, \dots, k\} $ we have that $T(I)_i = i$ and similarly for $J$. Since $p < \ell$, we have that none of $T(I)_1, \dots, T(I)_\ell$ or $T(J)_1, \dots, T(J)_\ell$ lie in $\{p+1, \dots, k\}$. So, without loss of generality, we may assume that $I \cap \{p+1, \dots, k\} = \emptyset$ and $J \cap \{p+1, \dots, k \} = \emptyset$. So we have that the entries of $T(I)$ and $T(J)$ are in ascending order.

For the set $J'$, since $T(J)_\ell = p < q < i_{s+d} = T(I)_{\ell+1}$, it follows that the entries of $J'$ are distinct and moreover the entries of tuple $T(J')$ are exactly in the order shown above in \eqref{eq:def:I'fori-1} in the definition of $J'$. Similarly for the set $I'$, we have that $T(I)_\ell = i_s < q = T(J)_{\ell+1}$ and so the entries of $I'$ above are distinct and the tuple $T(I')$ is identical to the order shown above. So, we have shown that $v_I + v_J = v_{I'} + v_{J'}$. Since $T(I')_{\ell} = i_s \neq p$ and $i_s < q$, and $T(I')_{\ell} > q$, we see that $\langle f, v_{I'}\rangle = 0$ and so it follows that $\langle f, v_{J'} \rangle = 0$. Therefore we are done for the polytope $P_{i-1}$ and the proof is complete.
\end{proof}

We are now ready to prove our main result.

\begin{proof}[\bf Proof of Theorem~\ref{thm:main_theorem}]
By Lemma~\ref{lemma gr fflv inner products}, we have that $\varphi^i(V(P_{i-1})) = V(P_i)$. The second part of Lemma~\ref{lemma gr fflv no bad edges}, implies that $\varphi^i(P_{i-1}) \subseteq P_i$. So see this, take any point $\varphi^i(x) \in \varphi^i(P_{i-1})$. We have $x = \sum_{v \in V(P_{i-1})} \alpha_v v$ for some $\alpha_v \in \RR_{\ge 0}$ such that $\sum_{v} \alpha_v = 1$. By the second part of Lemma~\ref{lemma gr fflv no bad edges}, we can rewrite this expression for $x$ so that it is supported on vertices which have non-negative or non-positive inner product with $f^i$. Since $\varphi^i$ acts linearly on each half-space defined by $(f^i)^\perp$, it follows that $\varphi^i(x)$ lies in $P_i$.
The first part of Lemma~\ref{lemma gr fflv no bad edges} implies that $\varphi^i(P_{i-1}) \supseteq P_{i}$. To see this, apply the above argument to the tropical map $(\varphi^i)^{-1} = \varphi_{-w, F}$ where $F = \conv\{0, f \}$. This completes the proof. 
\end{proof}

\bigskip

\bibliographystyle{abbrv}
\bibliography{References.bib}

\bigskip
\noindent
\footnotesize {\bf Authors' addresses:}

\medskip 

\noindent Department of Pure and Applied Mathematics, 
Osaka University, Suita, Osaka 565-0871, Japan\\
E-mail address: {\tt oliver.clarke.crgs@gmail.com}

\medskip\noindent
Department of Pure and Applied Mathematics,
Osaka University, Suita, Osaka 565-0871, Japan\\
E-mail address:  {\tt higashitani@ist.osaka-u.ac.jp}

\medskip\noindent
Department of Mathematics and Department of Computer Science, KU Leuven, Belgium 
\\
UiT – The Arctic University of Norway, 9037 Troms\o, Norway
\\ E-mail address: {\tt fatemeh.mohammadi@kuleuven.be}

\end{document}